\documentclass{amsart}

\usepackage{graphicx}
\usepackage{color}
\usepackage{enumerate}
\usepackage{accents}

\newtheorem{theorem}{Theorem}[section]
\newtheorem{lemma}[theorem]{Lemma}
\newtheorem{proposition}[theorem]{Proposition}

\theoremstyle{definition}
\newtheorem{definition}[theorem]{Definition}
\newtheorem{example}[theorem]{Example}

\newtheorem{corollary}[theorem]{Corollary}

\newtheorem{hypothesis}{Hypothesis}
\theoremstyle{remark}
\newtheorem{remark}[theorem]{Remark}

\numberwithin{equation}{section}

\newcommand*\interior[1]{#1^{\mathsf{o}}}

\renewcommand{\mod}{\text{ mod }}

\newcommand{\R}{\mathbb{R}}
\newcommand{\N}{\mathbb{N}}
\newcommand{\Z}{\mathbb{Z}}
\newcommand{\Y}{\mathbb{Y}}

\renewcommand{\a}{\mathbb{A}}
\renewcommand{\aa}{\a^{2}}

\newcommand{\2}{(\xi,x)}

\newcommand{\pp}{{ \phi^{+}}}
\newcommand{\pn}{{ \phi^{-}}}
\newcommand{\ps}{{ \phi^{*}}}
\newcommand{\ppn}{{ \phi^{\pm}}}

\newcommand{\tpp}{{\tilde \phi^{+}}}
\newcommand{\tpn}{{\tilde \phi^{-}}}
\newcommand{\tppn}{{\tilde \phi^{\pm}}}
\newcommand{\tp}{{\tilde P}}
\newcommand{\tcp}{{\tilde C^{+}}}
\newcommand{\tcn}{{\tilde C^{-}}}
\newcommand{\tcpn}{{\tilde C^{\pm}}}

\newcommand{\hs}{{\hat S}}
\newcommand{\hp}{{\hat \phi}}
\newcommand{\hpp}{{\hat \phi}^+}
\newcommand{\hpn}{{\hat \phi}^-}
\newcommand{\hppn}{{\hat\phi}^{\pm}}
\newcommand{\hP}{{\hat \Phi}}
\newcommand{\hf}{{\hat f}}
\newcommand{\hpi}{{\hat P}}
\newcommand{\hk}{\hat{\k}}

\newcommand{\f}{{f}_{\theta}}
\newcommand{\g}{{G}_{\theta}}

\newcommand{\Tt}{{\mathbb{T}^{2}}}
\newcommand{\T}{{\mathbb{T}^1}}
\newcommand{\I}{{\mathbb{I}}}

\newcommand{\cE}{\mathcal{E}}

\renewcommand{\k}{\mathrm{K}}

\newcommand{\m}{\mathcal{M}_{S}}

\newcommand{\C}{C^{1}}

\newcommand{\CCC}{C^{3}}

\newcommand{\de}{d^{\mathfrak{1}}}
\newcommand{\dze}{d^{\mathfrak{1}}}
\newcommand{\dz}{d^{\mathfrak{0}}}
\newcommand{\dzz}{d^{\mathfrak{0}}}

\newcommand{\q}{Q_{f}}
\newcommand{\qq}{Q'_{f}}

\newcommand{\clpn}{cl(\Phi^{\pm})}

\newcommand{\filpn}{[\Phi^{\pm}]}

\newcommand{\hclpn}{cl(\hP^{\pm})}

\newcommand{\hfilpn}{[\hP^{\pm}]}

\newcommand{\GA}{{\rm GA}}

\newcommand{\id}{{\rm id}}
\newcommand{\supp}{\operatorname{supp}}

\begin{document}

\title{Invariant Graphs for Chaotically Driven Maps}

\author[Fadaei]{S. Fadaei}
\address{Department of Mathematics, Ferdowsi University of Mashhad,
Mashhad, Iran.  }
\email{faddaei.262@gmail.com}

\author[Keller]{G. Keller}
\address{Department of Mathematics, University of Erlangen-Nuremberg, Cauerstr. 11, 91058 Erlangen, Germany}
\email{keller@math.fau.de}

\author[Ghane]{F. H. Ghane$^{*}$}
\address{Department of Mathematics, Ferdowsi University of Mashhad,
Mashhad, Iran.  }
\email{ghane@math.um.ac.ir}
\thanks{$^*$Corresponding author}






\begin{abstract}
This paper investigates the geometrical structures of invariant graphs of skew product systems of the form $F : \Theta \times \I \to \Theta \times \I , (\theta,y)\mapsto (S\theta,\f(y))$ driven by a hyperbolic base map $S : \Theta \to \Theta$
(e.g. a baker map or an Anosov surface diffeomorphism)
and with monotone increasing fibre maps $(f_{\theta})_{\theta \in \Theta}$ having negative Schwartzian derivatives.
We recall a classification, with respect to the number and to the Lyapunov exponents of invariant graphs, for this class of systems.
Our major goal here is to describe the structure of invariant graphs and study the properties of the
pinching set, the set of points where the values of all of the invariant graphs coincide.
In \cite{KO2}, the authors studied skew product systems driven by a generalized Baker map $S:\Tt\to\Tt$ with the restrictive assumption that $\f$ depend on $\theta=\2$ only through the stable coordinate $x$ of $\theta$. Our aim is to relax this assumption and construct a fibre-wise conjugation
between the original system and a new system for which the fibre maps depend only on the stable coordinate of the
derive.
\end{abstract}

\maketitle

\section{Introduction}
\subsection{Motivation and related works}\label{sub:mot}
The main objective in this paper is to describe the geometric structures of invariant graphs of a certain
class of skew products.
The existence of invariant graphs considerably simplifies the dynamics of the forced systems and they are currently object of intense study.

A skew product system is a dynamical system $(\Theta \times \Y , F )$ which can be written as
$$F : \Theta \times \Y \to \Theta \times \Y , \quad (\theta,y)\mapsto (S\theta,\f(y)).$$
Here the dynamics on the fibre space $\Y$ may be interpreted as being driven by another system
$(\Theta, S)$ since the transformations $\f : \Y \mapsto \Y$ depend on $\theta$. For example, $\theta$ can be used to induce
an additive or multiplicative external noise, i.e. $\f$ is of the form
\begin{itemize}
    \item $\f(y) = h(y) + g(y)$, or
    \item $\f(y) = g(\theta)h(y)$,
\end{itemize}
where $g(\theta)$ represents the random noise. On the other hand, the space $\Y$ is principally considered
as the fibre space over the basis dynamics $(\Theta , S)$, i.e. the fibre map $\f$ can be considered as a map
from $\{\theta\} \times \Y$ to $\{S\theta\} \times \Y$, where $\{\theta\} \times \Y$ is the fibre space over $\theta \in \Theta$. In fact, this natural
structure appears in innumerable examples of the dynamics which are relevant theoretically or for
applications. Furthermore, an invariant graph $\phi : \Theta \to \Y$ is a function that satisfies
$$F(\theta,\phi(\theta)) = (S\theta , \phi(S\theta)),$$
for all $\theta \in \Theta$.

In \cite{St1, St2}, Stark provided
conditions for the existence and regularity of invariant graphs
and discussed a number of applications to the filtering of time series, synchronization and quasiperiodically forced systems.
Campbell and Davies \cite{CD} proposed related results on the ergodic properties of attracting graphs and stability results for such graphs
under deterministic perturbations.
Invariant graphs have a wide variety of applications in many branches of nonlinear dynamics,
some knowledge of such applications is also available, through
works of several authors
(e.g. \cite{CD, DC, HOJ1, HOJ2, J1, K1, KO2, PC, SD} etc.).

In this paper, we study geometrical structures of the invariant graphs in a simple but important
case where $\Y = \I \subseteq \R$ and $\f$ is monotone increasing.
Besides, we focus on skew product systems whose monotone fibre maps $(f_{\theta})_{\theta \in \Theta}$ possess
negative Schwarzian derivatives. Here the last condition means that
$$\dfrac{f_{\theta}'''}{f_{\theta}'} - \frac{3}{2} \Big( \frac{f_{\theta}''}{f_{\theta}'} \Big) ^{2} < 0, $$
is satisfied for all $\theta \in \Theta$. Note that in our setting this condition guarantees that there do not exist more than three
invariant graphs, which especially allows us to study the bounding graphs, i.e. the lower and upper outermost graphs $\pn$ and $\pp$. In other words, with this technical tool, we can investigate the fine structure
of all of the invariant graphs that possibly appear. Moreover, we assume that the skew
product dynamic is in total compressing, in the sense that $F(\Theta \times \I) \subseteq \Theta \times \interior{\I}$. Our major goal here is to describe the structure of invariant graphs and study the properties of the pinching set, the set of points where the values of all of the invariant graphs coincide, mainly, in the case that the basis dynamical system is an Anosov diffeomorphism or a generalized
Baker map on the two-dimensional torus $\Tt = \R^2 / \Z^2$.\\
If $S : \mathbb{T}^2 \to \mathbb{T}^2$ is the Baker map and if instead of branches with negative
Schwarzian derivative one studies affine branches $\f(y) = \lambda y + \cos(2\pi x)$ with $\lambda \in
(1/2, 1)$ (where $\theta = \2$), it turns out that $\pp=\pn$ everywhere and that this
is the graph of a classical Weierstrass function - a H\"older-continuous but nowhere
differentiable function. These functions and their generalizations have received much
attention during the last decades, culminating in the recent result that the graph of this function has Hausdorff dimension $ 2 + \frac{\log \lambda}{\log 2} $ for all $\lambda \in (1/2, 1)$, i.e. in the range
of parameters for which the skew product system $F$ is partially hyperbolic \cite{B1, Sh}.
Related, but less complete results were also obtained for generalizations where
the cosine function is replaced by other suitable functions, where the Baker map is
replaced by some of its nonlinear variants, or where the slope $\lambda$ of the contracting branches is also allowed to depend on $x$, see e.g. \cite{P, L, O}. In all these
cases, however, the fibre maps remain affine. Other results deal with the local H\"older
exponents of these graphs \cite{B3, B2}.\\
Let $S: \T \to \T$ be an irrational circle rotation. Such quasi-periodically forced systems
were studied by J\"ager in \cite{J1, J3}. He gives the following result \cite{J1} under the additional
assumption that $(\theta, y) \mapsto \f'(y)$ is continuous: There are three possible cases:\begin{enumerate}[(1)]
\item There exists only one Lebesgue equivalence class of invariant graphs $\phi$, and $\lambda(\phi)\leq 0$. This equivalence class contains at least an upper and a lower semi-continuous
representative. 
\item
There exist two invariant graphs $\phi$ and $\ps$ with $\lambda(\phi)<0$ and $\lambda(\ps) = 0$. The
upper invariant graph is upper semi-continuous, the lower invariant graph
lower semi-continuous.
\item There exist three invariant graphs $\pn \leq \ps \leq \pp$  with $\lambda(\ppn) < 0$ and $\lambda(\ps) > 0$.   $\pn$ is lower semi-continuous and $\pp$ is upper semi-continuous. If $\ps$
is neither upper nor lower semi--continuous, then      $\supp(m_{\pn}) = \supp(m_{\ps}) \\= \supp(m_{\pp})$.
\end{enumerate}

All three cases occur. For case (iii) this was proved in \cite{J3}.
In this note, we replace the quasi-periodic base by a chaotic one, namely by an Anosov diffeomorphism or a
Baker transformation. Roughly speaking,
we shall assume that the dynamic is in total compressing and investigate the pinching set of
the outermost graphs. Interestingly, it turns out that the separating graph $\ps$ plays also a key role
in our study.
Also, A. Bonifant and J. Milnor studied in \cite{B4} skew product systems fairly similar to
ours, i.e. ones with a chaotic basis dynamical system whose fibre maps have negative Schwarzian
derivatives. In spite of the fact that the pictures of their skew product dynamics are significantly different from what we study here, there are some relations. Thus we give a brief remark on this point. Namely, they assumed that the two outermost invariant graphs $\pn$ and $\pp$ are constant, say
-1 and 1, and both of them are attracting simultaneously. More recently, G. Keller and A. Otani studied in \cite{KO2} the hyperbolic case where $S:\Tt\to\Tt$ is a generalized Baker map with the restrictive assumption that $\f$ depend on $\theta=\2$ only through the stable coordinate $x$, so that  $\f$ can be written as  $f_{x}$, and $x \mapsto f_{x}$ is continuous
on $\T$, where $\T$ is the one-dimensional torus. Also $f'_{x}>0$ and $\mathcal{S}f_{x} < 0$. In this paper we are going to get rid of this assumption that fibre dynamics depend on $\theta \in \Theta$ only through the stable coordinate of the drive.
\subsection{Contribution of this paper}
We are interested in
skew product systems with chaotic basis dynamical systems. In this paper, we concentrate on skew product systems which have a nice hyperbolic structure at the base.\par
In Section 2, Proposition \ref{theo:g} and Theorem \ref{Lem:f} , we construct a fibre-wise conjugation  between the original system and a new system  which satisfies the restrictive assumption that the fibre dynamics depend only on the stable direction of the drive. By applying this technical tool and  in the case that the basis dynamical system is an Anosov diffeomorphism, we describe the structure of the pinching set which consists of those points at which the global attractor is pinched to one point in the fibre spaces  (Theorem \ref{theo:anosov}): it is a union of complete global stable fibers. In Section 3, we focus on skew products driven by a generalized Baker transformation and study the structure of pinching set in Theorem \ref{theo:baker}.

The construction of the fibre maps that depend only on the stable direction of the drive is inspired by Sinai's \cite{Si3} construction of a function on a one-sided shift space that is cohomologous to a given one on a two-sided shift space,
see also \cite[Lemma~1.6]{Bo}.

\subsection{General setting}
We describe the general assumptions of this paper which specify
the basic concepts introduced in Subsection \ref{sub:mot}. Note that at the beginning, we only restrict the fibre maps and
leave the basis dynamical system arbitrary. After several general facts are discussed, we
shall consider more specific basis maps. More precisely,
we are interested in skew product dynamical systems $F : \Theta \times\I\rightarrow\Theta\times\I$, $F(\theta,y) = (S\theta , f_{\theta}(y))$, defined as below:
\begin{definition}
$\mathcal{F}$ denotes the family of all skew product transformations $F$ on $\Theta\times\I$ where
\begin{itemize}
\item $(\Theta, d)$ is a compact metric space and $S : \Theta \rightarrow \Theta $ is a Borel measurable invertible map.
\item$\I = [-M,M]$ is a compact interval and $\interior{\I}=(-M,M)$ where $M \in \mathbb{R}$.
\item The fibre maps $ f_{\theta} : \I \to \interior{\I} $ are given by $f_{\theta}(y) = \pi_{2} \circ F(\theta,y) $  with $\pi_{2}$ the natural projection from $\Theta \times [-M,M]$ to $[-M,M]$. Its derivative with respect to $y\in \I$ will be denoted by $f'_{\theta}$. We will assume that the fiber maps $f_{\theta} $ are uniformly bounded increasing $\CCC$-maps with negative Schwarzian derivative, i.e.
\begin{equation}\label{def:sch}
    \inf_{(\theta,y)\in \Theta \times \I}\f'(y) > 0,\quad\text{ and }\quad\sup_{(\theta,y)\in \Theta \times \I}\mathcal{S}\f(y) < 0   ,
\end{equation}
where $\mathcal{S}f_{\theta} := \dfrac{f_{\theta}'''}{f_{\theta}'} - \dfrac{3}{2} \Big( \dfrac{f_{\theta}''}{f_{\theta}'} \Big) ^{2}$.
\end{itemize}
\end{definition}
\begin{remark}\label{rem:com}
By definition, the dynamic compresses the whole space, i.e. $F(\Theta \times \I) \subset \Theta \times \interior{\I}$.
\end{remark}

For the iterates  $F^{n}$ of $F$ we adopt the usual notation  $ F^{n}(\theta,y ) = (S^{n}\theta , f_{\theta}^{n}(y))$  where $f^{n}_{\theta} = f_{S^{n-1}(\theta)} \circ \cdots \circ f_{\theta}$  . Hence  $\f^{n+k}(y) = f_{S^k \theta}^n (f^k_\theta(y))$. For $n = 1$ and $k =-1$ this includes the identity $\f^{-1}(y) = (f_{S^{-1}\theta})^{-1}(y)$.\\

\subsection{Bounding graphs, Lyapunov exponents, and pinch points}
In this subsection, we introduce the concepts and the notations which are
 basic for the study in the following sections.\\
Invariant graphs are fundamental objects in the study of skew product systems, and they are of
major interest.
\begin{definition}[Invariant graph]\label{def:graph}

Let $F \in \mathcal{F}$. A measurable function $\phi : \Theta \to \I$ is called an invariant graph (with respect to $F$) if for all $\theta \in \Theta$:
$$ F(\theta , \phi(\theta)) = (S \theta , \phi(S \theta )) ,    \text{ equivalently} \quad f_{\theta}(\phi(\theta)) = \phi(S\theta)  .
$$
The point set $\Phi := \{ (\theta , \phi(\theta))  : \theta \in \Theta \}$ will be called invariant graph as well, but it is labeled with the corresponding capital letter. Denote by $cl(\Phi)$ the closure of $\Phi$ in $\Theta \times \I$, and by $[\Phi]$ its filled-in closure in $\Theta \times \I$, i.e.
$$[\Phi]= \Big\{ (\theta,y) \in \Theta \times \I : \exists y_{1},y_{2} \in \I   \text{ s.t } (\theta,y_{1}), (\theta,y_{2}) \in cl(\Phi) , y_{1} \leq y \leq y_{2} \Big\} .$$
\end{definition}

\begin{definition}
Let $\phi^{\pm}_{n}(\theta) := f^{n}_{S^{-n}\theta} (\pm M)$ for $\theta \in \Theta$ and $n \in \N$. Define
$$\phi^{\pm}(\theta) := \lim_{n \rightarrow \infty} \phi^{\pm}_{n}(\theta).$$
We call $\pp , \pn$ the upper and the lower bounding invariant graph, respectively. In order to simplify notations put:  $\k_{n}(\theta) := [\phi_{n}^{-}(\theta), \phi_{n}^{+}(\theta)] $ and   $\k(\theta) := [\phi^{-}(\theta), \phi^{+}(\theta)]$.
\end{definition}

The limit exists and is measurable, because  $\phi_{n}^{-}(\theta) \nearrow \phi^{-}(\theta)$ and $ \phi_{n}^{+}(\theta) \searrow \phi^{+}(\theta)$ in view of the monotonicity of fibre maps (the invariance follows from the monotonicity and
continuity of the fibre maps). If the fibre maps are not monotone one can still define an upper and a lower bounding graph, but
then these graphs no longer have to be invariant (see \cite{G} for details).\par
Denote by $\m (\Theta)$ the set of all $S$-invariant probability measures on $\Theta$, and by $\mathcal{E}_S(\Theta)$ the ergodic ones among them. Let $\mu \in \m (\Theta)$. For any invariant graph $\phi$ denote $\mu_{\phi} := \mu \circ (\id_{\Theta} , \phi)^{-1}$. Since $\phi$ is an invariant graph, it is obvious that the new measure $\mu_{\phi}$ is an $F$-invariant measure.\par
In the investigation of skew product systems,
attracting invariant graphs are often useful characteristics, which we simply call attractors. Formally, we introduce the following definition. Note that the effect of the attraction is observed only
in the fibre space $\I$.
\begin{definition}
A point $(\theta,y) \in \Theta \times \I$ is attracting, if there is a constant $\delta >0$ such that
$$\lim_{n\to \infty}|\f^n(y)-\f^n(z)| = 0,$$
for all $z \in (y-\delta,y+\delta)$.
Furthermore, an invariant graph $\phi$ is called an attractor with
respect to the invariant probability $\mu$ if $\mu$-almost every point is attracting.
\end{definition}
\begin{definition}[Global attractor]\label{def:att}
Let $F \in \mathcal{F}$. Then the global attractor of this system is defined by
$$\GA= \bigcap_{n\in \N}F^{n}(\Theta \times \I).$$
\end{definition}
\begin{remark}
Note that the $\pn$ and $\pp$ are the functions representing the lower and upper bounding graphs of the set $\GA$. In particular, these are the same objects as what are called the bounding
graphs in \cite{J1}. Since any invariant graph can only exist inside GA, the bounding graphs $\pn$ and
$\pp$ are the minimal and maximal invariant graphs, respectively.
\end{remark}
\begin{remark}
If $\Theta$ is a compact metric space and in case that $S$ is a homeomorphism, it is evident that $\pn, \pp$ are lower and upper semi-continuous,
respectively, by construction, and the global attractor $\GA$ is compact.
\end{remark}

There is a third relevant invariant graph $\ps$ that satisfies $\pn \leq \ps \leq \pp$. To see the existence of such an object, we start to consider the average Lyapunov exponents in the
fibre direction.

\begin{definition}[Lyapunov exponent]\label{def:lypunov}
Let $F\in \mathcal{F}$ and $(\theta,y)\in \Theta \times \I$ .  If the limit
$$\lambda (\theta,y) := \lim_{n\to \infty} \frac{1}{n} \log (f_{\theta}^{n})'(y), $$
exists, it is called the normal Lyapunov exponent  of $F$ in $(\theta,y)$.
Furthermore, if $\mu \in \mathcal{E}_S(\Theta)$ and $\phi$ is a $\mu$-a.e. invariant graph with $\log f'_{\theta}(\phi(\theta)) \in \mathcal{L}^{1}_{\mu}$, then its (fibre) Lyapunov exponent w.r.t. $\mu$ is defined as
$$\lambda_{\mu}(\phi) := \int_{\Theta} \log f'_{\theta}(\phi(\theta)) d\mu(\theta).$$
\end{definition}
Note that by the Birkhoff ergodic theorem 
\begin{equation*}
\begin{aligned}
\lambda(\theta,\phi(\theta)) 
&= \lim_{n\to \infty} \frac{1}{n} \log (f_{\theta}^{n})'(\phi(\theta)) = \lim_{n\to \infty} \frac{1}{n} \sum_{k=0}^{n-1}\log f_{S^{k}\theta}'(f_{\theta}^{k}(\phi(\theta))\\
&= \lim_{n \to \infty}\frac{1}{n} \sum_{k=0}^{n-1}\log f_{S^{k}\theta}'(\phi(S^k\theta))
= \int_{\Theta} \log f_{\theta}'(\phi(\theta))d\mu (\theta)  \\
&=\lambda_{\mu}(\phi)
\end{aligned}
\end{equation*}
for $\mu$-a.e.~$\theta\in\Theta$.
So the average Lyapunov exponent of an invariant graph equals its point-wise Lyapunov exponent for $\mu$-a.e. $\theta\in \Theta$ .

The following proposition is a slight modification of some results in \cite{J1}, which describes all possible scenarios
in our skew product system in terms of the invariant graphs and the fibre Lyapunov exponents.
Its proof is based on the striking property of functions with negative Schwarzian derivatives that
the cross ratio distortion increases, if they are applied, (see \cite{MS} for the details of the cross ratio distortion). In addition, the reason that the Lyapunov
exponents of the lower and the upper bounding graphs are always non-positive is our compressing
assumption (Remark\ref{rem:com})
.

\begin{proposition}
$\phi^{-}$, $\phi^{+}$ are measurable invariant graphs defined everywhere as pullback-limits and $\GA = \{ (\theta,y) \in \Theta \times \I : \phi^{-}(\theta)  \leq y \leq \phi^{+}(\theta)\}.$ Furthermore, there is a measurable graph $\phi^{*} $ defined everywhere, satisfying $\phi^{-} \leq \phi^{*} \leq \phi^{+}$ and such that for $\mu \in \cE_S(\Theta)$ exactly one of the following three cases occurs:\begin{enumerate}[(1)]
\item $\phi^{-}=\phi^{*}=
\phi^{+}$   $\mu$-a.s. and $\lambda_{\mu}(\phi^{\pm}) \leq 0$.
\item
$ \phi^{-}=\phi^{*}<\phi^{+}$ $\mu$-a.s. and $\lambda_{\mu}(\phi^{+})<0=\lambda_{\mu}(\phi^{-})$
 or
$\phi^{-}<\phi^{*}=\phi^{+}$ $\mu$-a.s. and
$\lambda_{\mu}(\phi^{-})<0=\lambda_{\mu}(\phi^{+})$.
\item
$\phi^{-}<\phi^{*}<\phi^{+}$ $\mu$-a.s., $\lambda_{\mu}(
\phi^{-})<0,\lambda_{\mu}(\phi^{+})<0$, $\lambda_{\mu}(\phi^{*})>0$, and $\phi^{*}$ is $\mu$ a.s. invariant.
\end{enumerate}
In all three cases hold:\begin{itemize}
    \item
Each graph which is $\mu-a.s.$ invariant equals $\pn, \pp$  or $ \phi^{*}$ $\mu$-a.s. In particular, $\mu_{\pn}, \mu_{\pp}$ and $\mu_{\phi^{*}}$ are the only $F$-invariant ergodic probability measures that project to $\mu$.
\item
For $\mu$-a.e. $\theta \in \Theta$ and every $y \in [\pn(\theta), \pp(\theta)] \setminus \{\phi^{*}(\theta) \}$ \end{itemize}
\begin{equation}
\begin{split}
\lim_{n\to \infty}|\f^{n}(y) - \pn(S^{n}\theta)| &= 0    \text{ if } y<\phi^{*}(\theta),\\
\lim_{n\to \infty}|\f^{n}(y) - \pp(S^{n}\theta)| &= 0    \text{ if } y>\phi^{*}(\theta).
\end{split}
\end{equation}
\end{proposition}
\begin{corollary}
For $\mu$-a.e. $\theta \in \Theta$ and every $y\in \I$,
\[
\lambda(\theta,y)=\lim_{n\to\infty} \frac{1}{n} \log(\f^{n})'(y)=
\begin{cases}
\lambda_{\mu}(\pn) \quad &\text{ if } y<\ps(\theta)\\
\lambda_{\mu}(\ps)  \quad &\text{ if } y=\ps(\theta)\\
\lambda_{\mu}(\pp)   \quad &\text{ if } y>\ps(\theta).
\end{cases}
\]
\end{corollary}

The main goal of this paper is to investigate the properties of the pinch points set, which consist of those base points, over which the system synchronizes trajectories with different initial values on the same fibre.
\begin{definition}[pinch points]
Let$$P := \big\{\theta \in \Theta : \pp(\theta) = \pn(\theta)\} = \{ \theta \in \Theta : (\GA)_{\theta} \text{ is a singleton set} \big\},$$where $(\GA)_{\theta} := \big\{y \in \I : (\theta, y) \in \GA \big\}$.
\end{definition}


\section{Chaotically driven skew product systems with fibre maps
with negative Schwarzian derivative}
In this section we consider the case of a hyperbolic base transformation $S:\Theta\to\Theta$.  We attempt to compare such a system with fibre maps depending locally on both, stable and unstable coordinates to one with fibre maps depending locally only on the stable coordinate, and to transfer knowledge about the latter system to the original one.
\subsection{Preliminaries}
As we will see later we need to measure how far compositions of branches are from identity. In this regard, let $\mathcal{D}(\I) $ be the space of all increasing $C^{1}$ maps $g : I_{g} \subseteq \I \to \I$ with $g'>0$.
Let $J_{g} := g(I_{g})$.  Then the maps $g : I_{g} \rightarrow J_{g} $  are diffeomorphisms.
\begin{definition} \label{def:meter}
Let $g,h \in \mathcal{D}(\I)$ such that $I_{h} \subseteq I_{g}$.  Define:
\begin{equation*}
\begin{split}
\dzz(g,h)
&:=
\sup \{|g \circ h^{-1} (y) - y| : y \in J_{h} \}\\
&= \sup \{|g(x) - h(x)| : x \in I_{h} \},
\end{split}
\end{equation*}
and
\begin{equation*}
\begin{split}
\dze(g,h) &:= \dzz(g,h) + || g' - h' ||_{\infty}\\
&=\dzz(g,h) + \sup \{ |g'(x)-h'(x)| : x\in I_{h} \}.\\
\end{split}
\end{equation*}
\end{definition}
\begin{remark}
It is easy to see that in general, $\dz$ is not symmetric but, if $I_{g}
= I_{h}$ then $\dz(g,h) = \dz(h,g)$. Also,  $\dz (g , h) =  \dz (g \circ h^{-1}  , \id \mid_{J_{h }}) $.
\end{remark}
We make the following general assumptions to keep the technicalities at a moderate level.
\begin{hypothesis}\label{hypo:lip}
Define: $ f_{\bullet} : \Theta \to \mathcal{D}(\I) $, $\theta \mapsto f_{\theta} $. Assume $f_{\bullet}$ is Lipschitz w.r.t $\de$, i.e there exists a real constant $L$ such that for all $\theta , \theta' \in \Theta$
\begin{equation*}
\begin{split}
\de(f_{\theta},f_{\theta'})&\leq L d(\theta,\theta').
\end{split}
\end{equation*}
Any such $L$ is referred to as a Lipschitz constant for the function $f_{\bullet} $. In view of~(\ref{def:sch}), this implies that $\sup_{y\in \I} |\log \f'(y) - \log f'_{\theta'}(y)| \leq L' d(\theta,\theta')$ for constant $L'>0$ and all $\theta , \theta' \in \Theta$.
\end{hypothesis}

\begin{hypothesis}\label{hypo:factor}
Assume there is a piecewise expanding and piecewise $C^{1+}$ mixing Markov map\footnote{Here and in the sequel $C^{1+}$ means "$\C$ with H\"older continuous derivative" without specifying the H\"older exponent.} $\hs:\mathbb{T}^{1} \to \mathbb{T}^{1}$ with finitely many branches which is a factor of $S^{-1}$, i.e.
$$\hs \circ \Pi = \Pi \circ S^{-1} \text{ for some measurable } \Pi:\Theta \to \mathbb{T}^{1}.$$
More precisely we assume there is an injection $\sigma : \mathbb{T}^{1} \to \Theta$ which is H\"older continuous on monotonicity intervals of $\hs$, which satisfies $\Pi \circ \sigma = \id_{\mathbb{T}^{1}}$ , and which is such that each $\theta \in \Theta$ belongs to the local stable fibre of $\sigma \Pi \theta$ w.r.t $S^{-1}$ in the sense that:
\begin{equation}\label{eq:contr}
\exists C > 0  \ \exists \alpha \in (0,1) \ \forall \theta \in \Theta \ \forall n>0:  \ d(S^{-n}\theta , S^{-n}(\sigma\Pi\theta)) \leq C\alpha^{n}d(\theta,\sigma\Pi\theta).
\end{equation}
\end{hypothesis}

\begin{hypothesis}\label{hypo:Q}
Let  $\q := \sup_{(\theta,y) \in \Theta\times\I} f'_{\theta}(y)$ and $ \qq := \sup_{(\theta,y)\in\Theta\times\I}|(\log f '_{\theta}(y))'|$. Assume
\begin{eqnarray*}
\alpha \cdot \q < 1 \quad \text{and} \quad \qq < \infty.
\end{eqnarray*}
\end{hypothesis}

In the next example we describe how Anosov surface diffeomorphisms fit the general framework of this paper.
\begin{example}[Anosov surface diffeomorphism]\label{exp:anosov}
Let $\Theta = \Tt$ and let $S:\Tt\to\Tt$ be a $C^2$ Anosov diffeomorphism. Denote by $T_{\theta}\Theta=E^s(\theta)+E^u(\theta)$ the splitting of the tangent bundle over $\theta \in \Theta$ into its stable and
unstable subbundles.  Sinai \cite{Si1, Si2} constructed Markov partitions
for Anosov diffeomorphisms (a simpler treatment can be found in \cite{Bo}). As indicated in the proof of Lemma 3 in \cite{Po}, one can construct a $C^{1+}$ expanding Markov interval map $\hs:\T\to\T$ that is a factor of $S^{-1}$, i.e. $\hs\circ\Pi = \Pi\circ S^{-1}$ with the projection $\Pi:\Tt\to\T$ and  the injection $\sigma:\T\to\Tt$ (see  section 6.3 of \cite{K1} for more details).
\end{example}

\begin{proposition} \label{prop:inv}
If $S$ is continuous, the sets $\clpn$ and $\filpn$ are forward $F$-invariant.
\end{proposition}
\begin{proof}
Let $(\theta,y)\in \clpn $, then there are $\theta_{n} (n \in \mathbb{N})$ such that $\theta_{n} \mapsto \theta$ and $\phi^{\pm}(\theta_{n})\mapsto y$. According to the continuity of $S$ and Hypothesis \ref{hypo:lip},
$$F(\theta,y) = (S\theta,f_{\theta}(y)) = \lim_{n \mapsto \infty} (S\theta_{n},f_{\theta_{n}}(\phi^{\pm}(\theta_{n}))) =  \lim_{n \mapsto \infty} (S\theta_{n},\phi^{\pm}(S\theta_{n}))) \in \clpn.$$
The forward $F$-invariance of $\filpn$ is an immediate consequence.
\end{proof}

\begin{lemma}\label{lem:id}
Suppose $f,g,h \in \mathcal{D}(\I)$ , $I_{g} \subseteq I_{h} , I_{g} \subseteq I_{f}$ and $J_{h} \subseteq I_{f}$. If, for given $\epsilon , \epsilon' > 0$, we have  $\dzz(h,\id) \leq \epsilon$ and $  \dzz(f \circ g^{-1} , \id) \leq \epsilon'$ then $\dzz (f \circ h \circ g^{-1} , \id ) \leq \epsilon' + Q  \epsilon$,  where $Q= \sup_{y\in I_f}  f'(y)  .$
\end{lemma}
\begin{proof}
For every $y \in I_{h}$ , $\vert h(y) - y \vert \leq \epsilon$. Let $\Delta(y) = h(y) - y$, then, by the mean value theorem, for every $z \in J_{g}$ there is $\delta \in I_f$ such that
$$\vert f \circ h \circ g^{-1} (z) - z\vert =\vert f(g^{-1}(z)) + f'(\delta) \cdot \Delta(g^{-1}(z)) - z \vert \leq \epsilon' + Q  \epsilon ,$$
where $Q = \sup_{y\in I_f}  f'(y)  .$
By Definition \ref{def:meter},
 $$\dzz (f \circ h \circ g^{-1} , \id ) = \sup \{|f \circ h \circ g^{-1} (z) - z| : z \in J_{g} \} \leq \epsilon' + Q \epsilon.$$
\end{proof}
\begin{lemma}\label{lem:id2}
Suppose $h,h_{i,R},h_{i,L}\in\mathcal{D}(\I),i=1,\cdots,k$; $I_{h_{i,R}}\subseteq J_{h_{i-1,R}},i = 2,\cdots,k$; $J_{h_{i,L}} \subseteq I_{h_{i+1,L}}, i = 1,\cdots,k-1 $ ; $ I_{h_{1,R}} \subseteq I_{h} $ and $ J_{h} \subseteq I_{h_{1,L}} $. For given $\epsilon_{0} , \epsilon_{1} , \epsilon_{2} , \cdots , \epsilon_{k} $, if we have  $\dzz(h_{i,L} , h_{i,R}) \leq \epsilon_{i} , i = 1, \cdots , k $ and $\dzz(h,\id) \leq \epsilon_{0 }$,
then $$ \dzz(h_{k,L}  \cdots \circ h_{1,L}  \circ h \circ h_{1,R}^{-1} \cdots \circ h_{k,R}^{-1} , \id) \leq \epsilon_{k} + \sum_{j=0}^{k-1}(\prod_{i=j+1}^{k} Q_{i}) \epsilon_{j} ,  $$
where $ Q_{i} = \sup_{y\in I_{h_{i,L}}} h'_{i,L} (y) $ for $i= 1 , \cdots , k$.
\end{lemma}

\begin{proof}
For $k=1$ the result is obvious according to the previous lemma:
$$ \dzz(h_{1,L} \circ h \circ  h_{1,R}^{-1} , \id) \leq  \epsilon_{1} + Q_{1} \cdot \epsilon_{0} .$$
By induction we assume the result is correct for $k-1$. Then for $k$ we have:
\begin{equation*}
\begin{aligned}
&\dzz(h_{k,L}  \circ h_{k-1,L} \circ \cdots \circ h \circ \cdots \circ h_{k-1,R}^{-1} \circ h_{k,R}^{-1} , \id) \\
&\quad\leq  \epsilon_{k} +Q_{k}  (\epsilon_{k-1}+\sum_{j=0}^{k-2}(\prod_{i=j+1}^{k-1} Q_{i}) \epsilon_{j})
= \epsilon_{k} + \sum_{j=0}^{k-1}(\prod_{i=j+1}^{k} Q_{i}) \epsilon_{j}.
\end{aligned}
\end{equation*}
\end{proof}
\begin{lemma}\label{coro:id}
 Let $\theta \in \Theta$ and $\theta'$ belong to the local stable fibre of $\theta \in \Theta$, in the sense of (\ref{eq:contr}). For arbitrary $n,k$ we have
$$\dzz(f_{S^{-n}\theta}^{k} \circ (f_{S^{-n}\theta'}^{k})^{-1} , \id) \leq C'\alpha^{n-k}d(\theta,\theta') ,$$
where $C'=\frac{\alpha CL}{1-\alpha \q}$.
\end{lemma}
\begin{proof}
Since $\theta \to f_{\theta}$ is Lipschitz,
$$\dzz(f_{S^{-(n+1-i)}\theta} , f_{S^{-(n+1-i)}\theta'}) \leq Ld(S^{-(n+1-i)}\theta,S^{-(n+1-i)}\theta') \text{ for every } i=1,\cdots,k.$$
Since $\theta,\theta'$ are in the same local stable fibre, by Hypothesis \ref{hypo:factor}
$$d(S^{-(n+1-i)}\theta,S^{-(n+1-i)}\theta') \leq C \alpha^{n+1-i} d(\theta,\theta')\text{ for every } i=1,\cdots,k .$$
In order to apply Lemma \ref{lem:id2}, for every   $i = 1, \cdots , k$ put:\\
\begin{equation*}
\begin{split}
h&= \id ,
\\
 h_{i,L}&= f_{S^{-(n-i+1)}(\theta)} , \\
 h_{i,R}&= f_{S^{-(n-i+1)}(\theta')},\\
\epsilon_{i} &= C \cdot L \cdot d(\theta , \theta') \cdot \alpha^{n+1-i},\\
\epsilon_{0} &= 0.
\end{split}
\end{equation*}
On $K_k(S^{-(n-k)}\theta)=[\phi_{k}^{-}(S^{-(n-k)}\theta') , \phi_{k}^{+}(S^{-(n-k)}\theta')]$,
\begin{equation*}
\begin{split}
&\dzz(f^{k}_{S^{-n}\theta} \circ (f^{k}_{S^{-n}\theta'})^{-1} , \id )\\
&= \dzz( f_{S^{-(n-k+1)}\theta}
\circ  \dots \circ  f_{S^{-n}\theta} \circ (f_{S^{-n}\theta'})^{-1} \circ \dots \circ (f_{
S^{-(n-k+1)}\theta'})^{-1} , \id )\\
&\leq \sum_{j=0}^{k}  \q^{k-j} \epsilon_{j} = C \cdot L \cdot d(\theta , \theta') \cdot \sum_{j=1}^{k} \alpha^{n+1-j} \cdot \q^{k-j} \\
& \leq C \cdot L \cdot d(\theta , \theta') \cdot \alpha^{n-k+1} \cdot \sum_{i=0}^{\infty} (\alpha \q)^{i}\\
&= C'  \cdot d(\theta , \theta') \cdot \alpha^{n-k},
\end{split}
\end{equation*}
where $C'=\frac{\alpha CL}{1-\alpha \q}$.
\end{proof}
\subsection{Main results}

In the following proposition we introduce a
family of partial diffeomorphisms $\g$ between subintervals of $\I$. In Theorem \ref{Lem:f} we prove that it plays the role of a conjugacy between the old and the new system restricted to their respective global attractors. Throughout we assume that Hypotheses 1--3 are satisfied.
\begin{proposition}\label{theo:g}
 For each $\theta \in \Theta$ and each $\theta'$ belonging to the local stable fibre of $\theta \in \Theta$ in the sense of (\ref{eq:contr}), let
\begin{equation}\label{eq:Gn-def}
 G_{\theta, \theta' , n} : \k_{n}(\theta) \to \k_{n}(\theta') ,\quad
 y \mapsto f^{n}_{S^{-n}\theta'} \circ (f^{n}_{S^{-n}\theta})^{-1} (y).
\end{equation}
Then the limit
\begin{eqnarray}\label{eq:G2}
G_{\theta , \theta'} := \lim_{n\rightarrow \infty} G_{\theta,\theta',n}
\end{eqnarray}
 exists uniformly on $\k(\theta)$, where $\k_{n}(\theta)=[\phi_{n}^- (\theta)$, $ \phi_{n}^{+}(\theta)]$ and   $\k(\theta)= [\phi^{-}(\theta), \phi^{+}(\theta)]$.
 Moreover, $G_{\theta , \theta'}(\k(\theta))=\k(\theta')$, and there is a continuous function $H_{\theta,\theta'}:\k(\theta)\to\R$ such that 
$\lim_{n\to\infty}G'_{\theta,\theta',n}=H_{\theta,\theta'}$ uniformly on $\k(\theta)$, $G'_{\theta,\theta'}=H_{\theta,\theta'}$ on $\interior{\k(\theta)}$,
and this equality extends to the one-sided derivatives of $G_{\theta,\theta'}$ at
the endpoints of $\k(\theta)$. In this sense, $G_{\theta,\theta'}:\k(\theta)\to \k(\theta')$ is an orientation preserving diffeomorphism whenever $\theta\not\in P$.
\end{proposition}
\begin{proof}
Recall that $\phi_{n}^{-}(\theta) \nearrow \phi^{-}(\theta)$ and $ \phi_{n}^{+}(\theta) \searrow \phi^{+}(\theta)$. Therefore, for $k\geq 1$,  $$\k_{n+k}(\theta) \subset \k_{n}(\theta).$$
Also for every $k\geq 1,$
$$\dzz(G_{\theta , \theta'  , n} , G_{\theta, \theta' , n+k}) = \dzz(G_{\theta , \theta', n} \circ (G_{\theta , \theta' ,  n+k})^{-1}, \id \mid_{\k_{n+k}(\theta')} ).$$
\\
For every $y \in \k_{n+k}(\theta')$,
\begin{equation*}
\begin{split}
&G_{\theta, \theta' , n} \circ G_{\theta, \theta' , n+k}^{-1}(y)\\&= (f^{n}_{S^{-n}\theta'}  \circ (f^{n}_{S^{-n}\theta})^{-1} \circ f^{n+k}_{S^{-(n+k)}\theta} \circ (f^{n+k}_{S^{-(n+k)}\theta'})^{-1}) (y)\\
&=(f^{n}_{S^{-n}\theta'} \circ ({f}^{n}_{S^{-n}\theta})^{-1} \circ f^{n}_{S^{-n}\theta} \circ f^{k}_{S^{-(n+k)}\theta} \circ (f^{k}_{S^{-(n+k)}\theta'})^{-1} \circ (f^{n}_{S^{-n}\theta'})^{-1})(y)\\
&=
(f^{n}_{S^{-n}\theta'} \circ  f^{k}_{S^{-(n+k)}\theta} \circ (f^{k}_{S^{-(n+k)}\theta'})^{-1} \circ (f^{n}_{S^{-n}\theta'})^{-1})(y).
\end{split}
\end{equation*}
By Lemma \ref{coro:id},
$$\dzz(f_{S^{-(n+k)}\theta}^{k} \circ (f_{S^{-(n+k)}\theta'}^{k})^{-1} , \id) \leq C'\alpha^{n}d(\theta,\theta').$$
In order to apply Lemma \ref{lem:id} and Hypothesis \ref{hypo:Q}, let\\
\begin{equation*}
\begin{split}
h&= f^{k}_{S^{-(n+k)}\theta } \circ (f^{k}_{S^{-(n+k)}\theta'})^{-1} ,
\\
 f&=g= f^{n}_{S^{-n}\theta'} ,\\
\epsilon'&= 0 ,\\
\epsilon&= C'\alpha^{n}d(\theta,\theta') .
\end{split}
\end{equation*}
Then
\begin{equation*}
\begin{split}
\dzz(G_{\theta, \theta' , n} \circ G_{\theta, \theta' , n+k}^{-1} , \id )
&= \dzz(f^{n}_{S^{-n}\theta'}  \circ f^{k}_{S^{-(n+k)}\theta } \circ (f^{k}_{S^{-(n+k)}\theta'})^{-1} \circ (f^{n}_{S^{-n}\theta'})^{-1}  , \id)
\\
&\leq
\epsilon' + \q^{n}\epsilon \\
&= C'(\alpha \q)^{n}d(\theta,\theta').
\end{split}
\end{equation*}
Thus $\dzz(G_{\theta, \theta' , n} \circ G_{\theta, \theta' , n+k}^{-1} , \id ) \leq C'\cdot (\alpha \q)^{n}\cdot d(\theta,\theta')$ uniformly for all $k\geq0$ where $C'$ is the constant $\frac{\alpha CL}{1-\alpha \q}$.
We proved the sequence $G_{\theta, \theta' , n} = f^{n}_{S^{-n}\theta'} \circ (f^{n}_{S^{-n}\theta})^{-1}$ is uniformly Cauchy and consequently uniformly convergent. \\
We turn to the uniform convergence of the derivatives. Observe first that
\begin{equation*}
\begin{aligned}
\log G'_{\theta,\theta',n}(y) &= \log \big((f^{n}_{S^{-n}\theta'})'((f^{n}_{S^{-n}\theta})^{-1}(y)) \cdot ((f^{n}_{S^{-n}\theta})^{-1})'(y) \big)\\
&= \sum _{k=0}^{n-1} \log f'_{S^{k-n}\theta'}
(f^{k}_{S^{-n}\theta'} \circ (f^{n}_{S^{-n}\theta})^{-1}(y)) - \log (f^{n}_{S^{-n}\theta})'((f^{n}_{S^{-n}\theta})^{-1}(y))\\
& = \sum_{k=0}^{n-1} \log \frac{f'_{S^{k-n}\theta'}(f^{k}_{S^{-n}\theta'}\circ (f^{n}_{S^{-n}\theta})^{-1}(y))}{f'_{S^{k-n}\theta}(f^{k}_{S^{-n}\theta}\circ(f^{n}_{S^{-n}\theta})^{-1}(y))}.
\end{aligned}
\end{equation*}
Now let $n>m$ . Then
\begin{equation*}
\begin{aligned}
&\log G'_{\theta,\theta',n}(y)- \log G'_{\theta,\theta',m}(y) \\
&= \sum _{j=0}^{n-1} \log \frac{f'_{S^{j-n}\theta'}(f^{j}_{S^{-n}\theta'}((f^{n}_{S^{-n}\theta})^{-1}(y)))}{f'_{S^{j-n}\theta}(f^{j}_{S^{-n}\theta}((f^{n}_{S^{-n}\theta})^{-1}(y)))} - \sum _{k=0}^{m-1} \log \frac{f'_{S^{k-m}\theta'}(f^{k}_{S^{-m}\theta'}((f^{m}_{S^{-m}\theta})^{-1}(y)))}{f'_{S^{k-m}\theta}(f^{k}_{S^{-m}\theta}((f^{m}_{S^{-m}\theta})^{-1}(y)))} \\
&= \sum _{j=0}^{n-m-1} \log \frac{f'_{S^{j-n}\theta'}(f^{j}_{S^{-n}\theta'}((f^{n}_{S^{-n}\theta})^{-1}(y)))}{f'_{S^{j-n}\theta}(f^{j}_{S^{-n}\theta}((f^{n}_{S^{-n}\theta})^{-1}(y)))}\\ 
&+ \sum _{j=n-m}^{n-1} \log \frac{f'_{S^{j-n}\theta'}(f^{j}_{S^{-n}\theta'}((f^{n}_{S^{-n}\theta})^{-1}(y)))}{f'_{S^{j-n}\theta}(f^{j}_{S^{-n}\theta}((f^{n}_{S^{-n}\theta})^{-1}(y)))} 
- \sum _{k=0}^{m-1} \log \frac{f'_{S^{k-m}\theta'}(f^{k}_{S^{-m}\theta'}((f^{m}_{S^{-m}\theta})^{-1}(y)))}{f'_{S^{k-m}\theta}(f^{k}_{S^{-m}\theta}((f^{m}_{S^{-m}\theta})^{-1}(y)))} \\
&= \text{I}+\text{II},
\end{aligned}
\end{equation*}
where
\begin{equation*}
\begin{aligned}
\text{I} &= \sum _{j=0}^{n-m-1} \log \frac{f'_{S^{j-n}\theta'}(f^{j}_{S^{-n}\theta'}((f^{n}_{S^{-n}\theta})^{-1}(y)))}{f'_{S^{j-n}\theta}(f^{j}_{S^{-n}\theta}
((f^{n}_{S^{-n}\theta})^{-1}(y)))} ,\\
\text{II} & = \sum _{j=n-m}^{n-1} \log \frac{f'_{S^{j-n}\theta'}(f^{j}_{S^{-n}\theta'}((f^{n}_{S^{-n}\theta})^{-1}(y)))}{f'_{S^{j-n}\theta}(f^{j}_{S^{-n}\theta}((f^{n}_{S^{-n}\theta})^{-1}(y)))} - \sum _{k=0}^{m-1} \log \frac{f'_{S^{k-m}\theta'}(f^{k}_{S^{-m}\theta'}((f^{m}_{S^{-m}\theta})^{-1}(y)))}{f'_{S^{k-m}\theta}(f^{k}_{S^{-m}\theta}((f^{m}_{S^{-m}\theta})^{-1}(y)))}.
\end{aligned}
\end{equation*}
Consider $z=(f^{n}_{S^{-n}\theta})^{-1}(y)$,
\begin{equation*}
\text{I} = \sum _{j=0}^{n-m-1} \log \frac{f'_{S^{j-n}\theta'}(f^{j}_{S^{-n}\theta}(z))}{f'_{S^{j-n}\theta}(f^{j}_{S^{-n}\theta}(z))} + \log \frac{f'_{S^{j-n}\theta'}(f^{j}_{S^{-n}\theta'}(z))}{f'_{S^{j-n}\theta'}(f^{j}_{S^{-n}\theta}(z))}.
\end{equation*}
By Hypothesis \ref{hypo:lip}, $\theta \mapsto \log f'_{\bullet}$ is Lipschitz. Hence
\begin{equation*}
\text{I} \leq \sum_{j=0}^{n-m-1} L'd(S^{j-n}\theta,S^{j-n}\theta')+\qq  \cdot |f^{j}_{S^{-n}\theta'}(z)-f^{j}_{S^{-n}\theta}(z)|.
\end{equation*}
By Lemma \ref{coro:id} and Hypothesis \ref{hypo:factor},
\begin{equation*}
\begin{aligned}
\text{I} &\leq \sum_{j=0}^{n-m-1} L'C \alpha ^{n-j}d(\theta,\theta') + \qq C'\alpha^{n-j}d(\theta,\theta')\\
&\leq \sum_{j=0}^{n-m-1} C'' \alpha^{n-j} d(\theta,\theta')  \quad \text{where} \quad C'' =L'C+\qq C'\\
&= \sum_{l=m+1}^{n} C''  \alpha^{l} d(\theta,\theta ')\\
&\leq \sum_{l=m}^{\infty} C''  \alpha^{l} d(\theta,\theta ')\\
&= C'' \frac{\alpha^{m}}{1-\alpha}d(\theta,\theta').
\end{aligned}
\end{equation*}
Now consider (II). If we change the index $j$ to $k+n-m$ then
\begin{equation*}
\begin{aligned}
\text{II} &=   \sum _{k=0}^{m-1} \log \frac{f'_{S^{k-m}\theta'}(f^{k+n-m}_{S^{-n}\theta'}
((f^{n}_{S^{-n}\theta})^{-1}(y)))}{f'_{S^{k-m}\theta}(f^{k+n-m}_{S^{-n}\theta}
((f^{n}_{S^{-n}\theta})^{-1}(y)))}- \sum _{k=0}^{m-1} \log \frac{f'_{S^{k-m}\theta'}(f^{k}_{S^{-m}\theta'}
((f^{m}_{S^{-m}\theta})^{-1}(y)))}{f'_{S^{k-m}\theta}(f^{k}_{S^{-m}\theta}((f^{m}_{S^{-m}\theta})^{-1}(y)))}\\
&=  \sum _{k=0}^{m-1} \log \frac{f'_{S^{k-m}\theta'}((f^{k}_{S^{-m}\theta'}\circ f^{n-m}_{S^{-n}\theta'})
((f^{n-m}_{S^{-n}\theta})^{-1} \circ (f^{m}_{S^{-m}\theta})^{-1}(y)))}{f'_{S^{k-m}\theta}((f^{m-k}_{S^{k-m}\theta})^{-1}(y))}  \\
&- \sum _{k=0}^{m-1} \log \frac{f'_{S^{k-m}\theta'}((f^{k}_{S^{-m}\theta'}(f^{m}_{S^{-m}\theta})^{-1}(y)))}{f'_{S^{k-m}\theta}((f^{m-k}_{S^{k-m}\theta})^{-1}(y))}\\
&= \sum _{k=0}^{m-1} \log \frac{f'_{S^{k-m}\theta'}(({f}^{k}_{S^{-m}\theta'}\circ {f}^{n-m}_{S^{-n}\theta'})((f^{n-m}_{S^{-n}\theta})^{-1} \circ (f^{m}_{S^{-m}\theta})^{-1}(y)))}{f'_{S^{k-m}\theta'}(f^{k}_{S^{-m}\theta'}((f^{m}_{S^{-m}\theta})^{-1}(y)))} .
\end{aligned}
\end{equation*}
Let $t=(f^{m}_{S^{-m}\theta})^{-1}(y) $ and $t'=(f^{n-m}_{S^{-n}\theta'})\circ (f^{n-m}_{S^{-n}\theta})^{-1}(t)$. According to Lemma \ref{coro:id},
$$ \mid t - t' \mid  \leq C'  \cdot \alpha^{m} \cdot d(\theta,\theta'),$$
and observing Hypothesis \ref{hypo:Q}, this implies
$$\mid f^{k}_{S^{-m}\theta'}(t) - f^{k}_{S^{-m}\theta'}(t') \mid \leq C' \cdot \alpha^{m} \cdot \q^{k} \cdot d(\theta , \theta').$$
So by the same hypothesis we have
\begin{equation*}
\begin{aligned}
\text{II} &= \sum _{k=0}^{m-1} \log \frac{f'_{S^{k-m}\theta'}(f^{k}_{S^{-m}\theta'}(t'))}{f'_{S^{k-m}\theta'}(f^{k}_{S^{-m}\theta'}(t))} \leq \sum_{k=0}^{m-1} \qq \cdot C' \cdot \alpha^{m} \cdot \q^{k}  \cdot d(\theta,\theta')\\
& \leq \frac{\qq C'}{\q-1}(\alpha \q)^{m}d(\theta,\theta').
\end{aligned}
\end{equation*}
Hence
\begin{equation*}
\begin{aligned}
\text{I}+\text{II}
\leq C''  \frac{\alpha^{m}}{1-\alpha}d(\theta,\theta ')+\frac{\qq C'}{\q-1}(\alpha \q)^{m} d(\theta,\theta'),
\end{aligned}
\end{equation*}
where $\alpha Q_f < 1$ by Hypothesis \ref{hypo:Q}. This shows that $\{\log G'_{\theta,\theta',n}\}$ is a uniform Cauchy sequence, so that
$H_{\theta,\theta'}:=\lim_{n\to\infty} G'_{\theta,\theta',n}$ exists uniformly on $\k(\theta)$ and
$G_{\theta,\theta',n}$ converges in $\de$-distance to $G_{\theta,\theta'}$. See e.g Theorem 7.17 of \cite{R}. 
As $f_\theta'>0$ for all $\theta$, we have $H_{\theta,\theta'}\geq0$, so that $G_{\theta,\theta'}$ is orientation preserving.

Finally, We prove that $G_{\theta,\theta'}(\k(\theta))=\k(\theta')$. Let $x\in\k(\theta)$.  According to \eqref{eq:G2}, $G_{\theta,\theta'}(x)=\lim_{n\to\infty}y_n$, where $y_n:=G_{\theta,\theta',n}(x)$. On the other hand since $\k_n(\theta) \searrow \k_{\theta}$ then,
$$y_n\in G_{\theta,\theta',n}(\k(\theta)) \subseteq G_{\theta,\theta',n}(\k_n(\theta))=\k_n(\theta').$$
Hence
$$G_{\theta,\theta'}(x)=\lim_{n\to\infty}y_n\in\k(\theta').$$
So $G_{\theta,\theta'}(\k(\theta))\subseteq\k(\theta')$. The inverse is obvious by interchanging the role of $\theta$ and $\theta'$ and observing that
$G_{\theta',\theta}=\lim_{n\to\infty}G_{\theta',\theta,n}=\lim_{n\to\infty}G_{\theta,\theta',n}^{-1}=G_{\theta,\theta'}^{-1}$ on $\k(\theta')$.
\end{proof}
\begin{definition}According to Hypothesis~\ref{hypo:factor},  $ \sigma\Pi\theta$ and $\theta$ are in the same local stable fibre. In order to simplify the notations define:
\begin{eqnarray}\label{eq:G-def}
G_{\theta,n} := G_{\theta,\sigma \Pi \theta,n}, \ \text{and} \ G_{\theta} := G_{\theta,\sigma \Pi \theta}.
\end{eqnarray}
\end{definition}
\begin{definition}
Suppose  $\phi$ is  an invariant graph.
Define :
$$
\hp : \Theta \to \I , \theta \mapsto G_{\theta} (\phi(\theta)).
$$
\end{definition}

\begin{lemma}\label{lem:hp}
$\hppn:\Theta \mapsto \I$ depend on $\theta$ only through $\Pi \theta$.
\end{lemma}
\begin{proof}
By Proposition \ref{theo:g}, $\g : \k(\theta)\to\k(\sigma\Pi\theta)$ is an orientation preserving homeomorphism for each $\theta \in \Theta$, so
\begin{equation}\label{eq:hatphi}
\hp^{\pm}(\theta) =\g (\phi^{\pm}(\theta))  =\phi^{\pm}(\sigma \Pi \theta).
\end{equation}
\end{proof}

For $\theta\in\Theta$ denote $\hk(\theta):=[\hpn(\theta),\hpp(\theta)]=\k(\sigma\Pi\theta)$ and let
\begin{equation}\label{eq:hf-def}
\hf_\theta:=f_{S^{-1}(\sigma\Pi S\theta)} \circ G_{\sigma\Pi \theta,S^{-1}\sigma\Pi S\theta}.
\end{equation}
This is well defined, because
 \begin{equation*}
 \begin{aligned}
\sigma\Pi(S^{-1}\sigma\Pi S\theta ) 
&= \sigma \hs (\Pi\sigma)\Pi S\theta  = \sigma \hs \Pi S\theta = \sigma \Pi S^{-1} S\theta = \sigma \Pi \theta
=\sigma(\Pi\sigma)\Pi\theta\\
&=\sigma\Pi(\sigma\Pi\theta),
\end{aligned}
\end{equation*}
so that,
by Hypothesis \ref{hypo:factor}, $\sigma\Pi(S^{-1}\sigma\Pi S\theta )$ and $\sigma\Pi\theta$ are in the same local stable fibre.

\begin{theorem}\label{Lem:f}
For each $\theta\in\Theta$,
\begin{equation}\label{eq:h-transport}
\hf_\theta(\hk(\theta))=\hk(S\theta),
\end{equation}
\begin{equation}\label{eq:conj}
f_{\theta}= G_{S\theta}^{-1} \circ \hf_{\theta} \circ G_{\theta},
\end{equation}
and the map $\hf_\theta$ 
depends on $\theta$ only through $\Pi(S\theta) \in \Theta$. (This is the point which represents the local stable fibre through $\theta$, see Hypothesis~\ref{hypo:factor}.) If $\theta\not\in P$, then
$\hf_\theta:\hk(\theta)\to\hk(S\theta)$ is an orientation preserving diffeomorphism.
\end{theorem}

\begin{proof}
In view of \eqref{eq:hf-def}
\begin{equation*}
\begin{aligned}
\hf_\theta(\hk(\theta))
&=
f_{S^{-1}(\sigma\Pi S\theta)} \circ G_{\sigma\Pi \theta,S^{-1}\sigma\Pi S\theta}(\k(\sigma\Pi\theta))
=
f_{S^{-1}(\sigma\Pi S\theta)}(\k(S^{-1}\sigma\Pi S\theta))\\
&=
K(\sigma\Pi S\theta)
=
\hk(S\theta),
\end{aligned}
\end{equation*}
which is \eqref{eq:h-transport}.
If $\theta\not\in P$, then
$\hf_\theta:\k(\sigma\Pi\theta)\to\k(\sigma\Pi S\theta)$ is an orientation preserving diffeomorphism by Proposition~\ref{theo:g}.

As $\sigma\Pi\theta=\sigma\Pi S^{-1}(S\theta)=\sigma\hs(\Pi S\theta)$,
a look at \eqref{eq:hf-def} reveals that
the map $\hf_\theta$ 
depends on $\theta$ only through $\Pi(S\theta) \in \Theta$.

Let $\hf_{\theta,n}:=f_{S^{-1}(\sigma\Pi S\theta)} \circ G_{\sigma\Pi \theta,S^{-1}\sigma\Pi S\theta,n}$, and observe that $\lim_{n\to\infty}\hf_{\theta,n}=\hf_\theta$ uniformly on $\hk(\theta)$.
 By \eqref{eq:G-def} and
 (\ref{eq:Gn-def}) we have
\begin{equation*}
\begin{aligned}
\hf_{\theta,n} \circ G_{\theta,n}
& = 
f_{S^{-1}(\sigma\Pi S\theta)} \circ G_{\sigma\Pi \theta,S^{-1}\sigma\Pi S\theta,n}
\circ G_{\theta,\sigma\Pi\theta,n}\\
&=
f_{S^{-1}(\sigma\Pi S\theta)} \circ G_{\theta,S^{-1}\sigma\Pi S\theta,n}
\\
&=
f^{n+1}_{S^{-n}(S^{-1}\sigma\Pi S \theta)} \circ (f^{n}_{S^{-n}\theta})^{-1}\\
& = G_{S\theta,n+1} \circ f_{\theta}.
\end{aligned}
\end{equation*}
Hence
\begin{equation}\label{eq:fog}
\begin{aligned}
\dz(\hf_{\theta} \circ \g , G_{S\theta} \circ f_{\theta})
&\leq
 \dz (\hf_{\theta} \circ \g , \hf_{\theta}\circ G_{\theta ,n})+
  \dz (\hf_{\theta} \circ G_{\theta,n} , \hf_{\theta,n}\circ G_{\theta ,n})\\
&\hspace*{2cm}  +
\dz(G_{S\theta ,n+1}
\circ f_{\theta}, G_{S\theta}\circ f_{\theta}).
\end{aligned}
\end{equation}
According to Definition \ref{def:meter}, if $f,g,h \in\mathcal{D}(\I) , I_{g} \subseteq I_{f} , J_{h}\subseteq I_{g} , J_{h}\subseteq I_{f}$, one can see easily
\begin{eqnarray*}\label{eq:met}
\dz(f\circ h,g\circ h) \leq \dz(f,g).
\end{eqnarray*}
So the right hand side of (\ref{eq:fog}) is less than
\begin{equation*}
\dz(\hf_{\theta} \circ \g , \hf_{\theta} \circ G_{\theta,n}) + \dz(\hf_{\theta} , \hf_{\theta,n} ) + \dz(G_{S\theta ,n+1}
, G_{S\theta}).
\end{equation*}
In view of Proposition \ref{theo:g} and the continuity of $\hf_{\theta}$, the last sum tends to zero as $n\to \infty$. This means
$$\dz(\hf_{\theta} \circ \g , G_{S\theta} \circ f_{\theta}) = 0, $$
and consequently
$$f_{\theta}= G_{S\theta}^{-1} \circ \hf_{\theta} \circ G_{\theta}.$$
\end{proof}

\begin{lemma}
$\hp^{\pm}:\Theta \mapsto \I$ are invariant graphs for the family $(\hf_\theta)_{\theta\in\Theta}$ in the sense that
$$\hf_{\theta}(\hp^{\pm}(\theta)) = \hp^{\pm}(S (\theta)).$$
\end{lemma}
\begin{proof}
By (\ref{eq:conj}), 
$$\hf_{\theta}(\hp^{\pm}(\theta)) = \hf_{\theta}(\g(\phi^{\pm}(\theta))) =  G_{S\theta} (f_{\theta}(\phi^{\pm}(\theta)) = G_{S\theta} (\phi^{\pm}(S\theta)) = \hp^{\pm}(S\theta).$$
\end{proof}

\begin{remark}
Lyapunov exponents for the family $(\hf_\theta)_{\theta\in\Theta}$ are defined as those for the family $(f_\theta)_{\theta\in\Theta}$ in Definition~\ref{def:lypunov}:
$$\hat\lambda (\theta,y) := \lim_{n\to \infty} \frac{1}{n} \log (\hf_{\theta}^{n})'(y) $$
whenever this limit exists, and
$$\hat\lambda_{\mu}(\hat\phi) := \int_{\Theta} \log \hf'_{\theta}(\hat\phi(\theta)) d\mu(\theta)$$
when $\mu \in \mathcal{E}_S(\Theta)$ with $\mu(P)=0$ and $\hat\phi$ is a graph invariant $\mu$-a.e.~for the family $(\hf_\theta)_{\theta\in\Theta}$ with $\log \hf'_{\theta}(\hat\phi(\theta)) \in \mathcal{L}^{1}_{\mu}$. The condition $\mu(P)=0$ guarantees that $\hf'_{\theta}(\hat\phi(\theta))$ is well defined for $\mu$-a.a~$\theta$.
\end{remark}

\begin{lemma}
The Lyapunov exponents of $\hppn$ coincide with those of $\ppn$, namely
$$\lambda_{\mu}(\hppn)  = \lambda_{\mu}(\ppn) = \int \log \f'(\ppn(\theta)) d\mu(\theta).$$
\end{lemma}
\begin{proof} 
The Lyapunov exponents of $\hppn$ are equal to
\begin{equation*}
\begin{aligned}
\lambda_{\mu}(\hppn)
&= \int\log |\hf_{\theta}'(\hppn(\theta))|d\mu(\theta)\\
&= \int \big(\log |G'_{S\theta}(f_{\theta}\circ G_{\theta}^{-1}(\hppn(\theta)))| + \log| f_{\theta}'(G_{\theta}^{-1}(\hppn(\theta)))|\\
&\hspace*{5.5cm} - \log| G'_{\theta}(G_{\theta}^{-1}(\hppn(\theta)))|\big)d\mu(\theta)
\\
&= \int (\log |G'_{S\theta}(\ppn (S\theta))| + \log| f_{\theta}'(\ppn(\theta))|- \log| (G'_{\theta}(\ppn(\theta))|)d\mu(\theta)
\\&=\lambda_{\mu}(\ppn).
\end{aligned}
\end{equation*}
\end{proof}

\begin{definition}
Denote by  $\hclpn$ the closure of the graph of $ \hp^{\pm}$ in $\Theta\times\I$ and by $\hfilpn$ its filled-in closure in $\Theta\times\I$.
Let $\hat{P} := \{ \theta \in \Theta : \hp^{+}(\theta) = \hp^{-}(\theta) \}$ be the set of pinch points of the new system, and denote by $\hat{C}^{\pm} \subseteq \Theta$ the set of continuity points of $\hp^{\pm}$.
\end{definition}

\begin{proposition} \label{prop:pset}
1) $P=\hat{P}$.\\
2)  $P$ is forward and backward $S$-invariant.\\
3) $\Pi P$ is forward and backward $\hs$-invariant.
\end{proposition}
 \begin{proof}
1) According to Proposition \ref{theo:g},  $\g:\k(\theta)\to\k(\sigma\Pi\theta)$ is a homeomorphism. So, the end points of $\k(\theta)$ are equal if and only if their images are equal, hence by Lemma \ref{lem:hp}, and definition of $P$ and $\hat{P}$ it is clear that $P=\hat{P}$.\\
2) $f_{\theta}$ is invertible so
$$\theta \in P \Leftrightarrow \f(\phi^{-}(\theta)) = \f(\phi^{+}(\theta))  \Leftrightarrow \phi^{-}(S\theta)=\phi^{+}(S\theta)  \Leftrightarrow  S\theta \in P .$$\\
3) Let $u \in \hs^{-1}(\Pi P)$, then there exist $\theta \in P$ and $\theta' \in \Theta$ such that $\hs u = \Pi \theta , u = \Pi \theta'$. So
$$\sigma \Pi \theta = \sigma \hs u = \sigma \hs \Pi \theta' = \sigma \Pi S^{-1} \theta'.$$
On the other hand,  $\theta \in P$ if and only if $\sigma \Pi \theta \in P.$ Indeed
$$\theta \in P \Leftrightarrow \theta \in \hat{P} \Leftrightarrow \hp^{-}(\theta)=\hp^{+}  (\theta) \Leftrightarrow \phi^{-}(\sigma \Pi \theta)=\phi^{+}(\sigma \Pi \theta)\Leftrightarrow \sigma \Pi \theta \in P.$$
So according to part 2,
$$\theta \in P  \Leftrightarrow \sigma \Pi \theta \in P  \Leftrightarrow \sigma \Pi (S^{-1}\theta') \in P  \Leftrightarrow S^{-1}\theta' \in P  \Leftrightarrow \theta' \in P.$$
This shows that $u \in \Pi P $ and so $\hs^{-1}\Pi P \subseteq \Pi P$. Conversely, assume $u \in \Pi P$. Then there
exists $\theta' \in P$ such that $u=\Pi\theta'$. So
\begin{eqnarray}
 \hs u = \Pi(S^{-1}\theta') \in \Pi(S^{-1}P) = \Pi P.
\end{eqnarray}
Hence $\hs(\Pi P) \subseteq \Pi P$.
\end{proof}
We see in the next theorem how this idea can be applied when the base map is an Anosov surface diffeomorphism.
\begin{theorem}\label{theo:anosov}
Let $\Theta = \Tt$ and let $S:\Tt\to\Tt$ be a $C^2$ Anosov diffeomorphism. Then  the set of pinch points consists of complete global stable fibre.
\end{theorem}
\begin{proof}
By Proposition \ref{prop:pset}\,(1), for the arbitrary Markov partition  $\{R_1,\cdots,R_n\}$, the set $P$ is a union of local stable fibres (i.e.~of connected components of intersections of global stable fibers with Markov rectangles). Hence, by part (2) of the same proposition,  the set $P$ consists of complete global stable fibre.
\end{proof}
\section{Negative Schwarzian branches driven by a Baker map}
In order to deal with the set of discontinuity points at the base, and keep the technicalities at minimum, we consider Baker transformations. These maps are bijective  with discontinuity at $a \in (0,1)$, i.e.
\[
S(\xi,x) = \begin{cases}
(\tau(\xi) , ax) & \text{if } \xi \in [0,a]\\
(\tau(\xi) , a+(1-a)x) &\text{if } \xi \in [a,1]
\end{cases},
\]
with
\[
\tau(\xi) =
\begin{cases}
a^{-1} \xi & \text{if } \xi \in [0,a)\\
(1-a)^{-1}(\xi-a) & \text{if } \xi \in [a,1)
\end{cases}.
\]

\begin{example}
Let $\tau(x) = 2x \mod 1, \f(y) = \arctan(ry) + \epsilon \cos(2\pi (x+\xi))$ , where
$r = 1.1$ and $\epsilon \in [0, 0.1]$. Then, choosing $M=0.86$ and $\I=[-M,M]$, one checks that $f_\theta(\I)\subseteq[-0.858, 0.858] \subset \interior{\I}$ for all $\theta\in\Theta=\Tt$. Figure 1 shows that in this case  we loose the continuity of upper  and lower bounding graphs.
\end{example}

\begin{figure}
\centering
\includegraphics[scale=0.4]{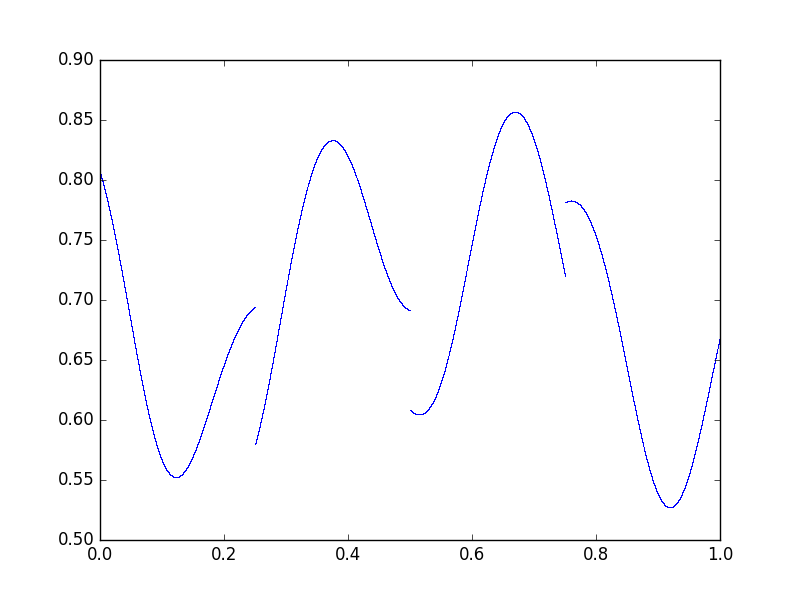}
\caption{discontinuity of $\phi_2^+$, where the parameters are $\epsilon =0.1 ,r = 1.1 $}
\end{figure}
\begin{remark}\label{baker0}
If $S$ is Baker transformation with discontinuity at $a$,
then  the sets $\clpn$ and $\filpn$ are nearly forward $F$-invariant from the measure-theoretic point of view in the sense that 
$F(\clpn) \setminus \clpn \subseteq \{ (0, x , y ) \in \Theta \times \I : 0 \leq x < 1, y \in \I \}$.
\end{remark}
We often will identify points $x \in \T = \R / \Z$ with points $x \in \a := [0, 1)$. In this case we also
identify $\2 \in \Tt$ with $\2 \in \aa.$
\begin{remark}\label{baker}
Notice that an important point in the definition of $S$ is its symmetry. That is, if $S(\xi,x)=(\tau (\xi), \rho_{\kappa(\xi)}(x))$, then it is easily calculated that the
inverse map has a symmetric expression, i.e.
\begin{equation}\label{ba}
 S^{-1}(\xi,x)=(\rho_{\kappa(x)}(\xi),\tau (x)).
\end{equation}
Moreover, if we define $\Pi: \mathbb{A}^2 \to \mathbb{A}$ by $\Pi(\xi,x)=x$ then Lemma \ref{lem:hp} states that $\hat{\phi}^{\pm}$ depend only on $x$, so that $\hat{P}$ is a union of fibers $\mathbb{A} \times \{z\}$.
\end{remark}

\begin{definition}
Define $L_{n}$ as the set of all discontinuity points of $S^{-n}$,
$$L_{n} :=\bigcup_ {k=0}^{n-1}    S^{-k} (\a \times  \{a\}),$$
and $$ L := \bigcup_{n=1}^{\infty} L_{n} = \bigcup_{k=0}^{\infty} S^{-k} (\a \times  \{a\}).$$
Also define
$$\tpp (\theta) := \lim_{r \mapsto 0 }\sup\{ \pp(\theta') : d(\theta' , \theta) < r \},$$
$$\tpn (\theta) := \lim_{r \mapsto 0 }\inf\{ \pn(\theta') : d(\theta' , \theta) < r \}.$$
Denote by $\tp = \{ \theta \in \Theta : \tpp (\theta) = \tpn (\theta) \}$ the set of pinch points and $\tcpn$ the set of continuity points  of $\tppn$.
\end{definition}
\begin{remark}
\begin{enumerate}
\item
The sets $L_{n}$ are finite unions of horizontal segments $\a \times  \{z\}$, and $L$ is a countable union of such segments.
\item $\mu (L_{n}) = 0$ for each $n$ and each $S$-invariant probability measure $\mu$.\item
$\phi_{n}^{\pm}$ is continuous at all points $\theta \in \Theta \backslash  L_{n} $ . \item
$\tpp$ is upper semi-continuous and $\tpn$ is lower semi-continuous and $\tpn \leq \pn \leq \pp \leq \tpp $. Thus, they are both Baire 1 functions. Consequently,   $\tcp , \tcn$ and also their intersection is residual.
\end{enumerate}
\end{remark}
\begin{lemma}\label{Lem:inv graph}
For every $\theta \in \Theta \backslash L$ we have  $\tppn (\theta) = \ppn (\theta)$ .
 \end{lemma}
\begin{proof}
By definition, $\pp(\theta) \leq \tpp(\theta)$ for all $\theta \in \Theta$. Let $\epsilon > 0 $ and $\theta \in \Theta \backslash L$. There exists $n>0$ such that $\pp (\theta) \geq \phi_{n}^{+}(\theta) - \epsilon$. As $\theta \notin  L$, there exists $r>0$ such that for every $\theta' \in \Theta$ with  $d(\theta',\theta) < r $, then $|\phi_{n}^{+}(\theta')-\phi_{n}^{+}(\theta)| < \epsilon$. Hence for every $\theta' \in B_{r}(\theta)$ :
\begin{equation}
\pp(\theta)  \geq \phi_{n}^{+}(\theta) - \epsilon  \geq \phi_{n}^{+}(\theta') - 2 \epsilon \geq \phi^{+} (\theta') - 2 \epsilon.
\end{equation}
Hence
$$\sup \{\phi^{+}(\theta') : d(\theta',\theta) < r \} \leq \pp (\theta) + 2\epsilon .$$
And so
$$\tpp(\theta) \leq \pp (\theta) + 2 \epsilon.$$
\end{proof}
\begin{corollary}
$\tp \subseteq P \subseteq \tp \cup L. $
\end{corollary}

\begin{lemma}\label{sec}
$(\tcp \cap \tcn) \backslash \tp \subseteq \Theta \backslash \overline{\tp}.$
\end{lemma}
\begin{proof}
Let $\theta \in \tcp \cap \tcn$ but $\theta \notin \tp$ . According to the definition of $\tppn$ there exists $r>0$ such that for every $\theta' \in B_r(\theta)$, $\tpp(\theta') > \tpn(\theta')$. This means that $B_{r} (\theta) \cap \tp = \emptyset$, and therefore  $\theta \notin \overline{\tp}$.
\end{proof}
\begin{theorem}\label{theo:baker}
Either $P \subseteq L $ or $\tcp \cap \tcn = \tp \subseteq P$.
\end{theorem}
\begin{proof}
Recall from Proposition \ref{prop:pset} and Remark \ref{baker} that $P = \hpi$ and $\hpi$ is a union of segments $\a \times \{z\}$. Hence, if $P \not\subseteq L $, there exists $z \in \a$ such that $(\a \times \{z\} ) \subseteq P$ but $( \a \times \{ z\}) \cap L = \emptyset$. Since $P$ is $S$-invariant and $P \backslash L$ contains the dense set $\bigcup_{k=0}^{\infty} S^{k}(\a \times \{ z\}) $, so $P \backslash L$ is dense.
As $P \backslash L \subseteq \tp$ in view of the last corollary, also $\tp$ is dense, i.e  $\overline{\tp} = \Theta$ . Then Lemma \ref{sec} implies $\tcp \cap \tcn \subseteq \tp$. The converse inclusion $\tp \subseteq \tcp \cap \tcn$ follows at once from the semi-continuity of $\tppn$.
\end{proof}
\begin{corollary}\label{coro:baker}
Either $P \subseteq L$ or $P$ is residual.
\end{corollary}

\bibliographystyle{amsplain}

\end{document}